\documentclass[11pt, letter]{amsart}

\usepackage{amsmath,amssymb,amsthm}
\usepackage{enumitem}
\usepackage{xcolor}
\usepackage{tikz-cd}
\usepackage{overpic}

\usepackage{comment}

\usepackage{caption}
\usepackage{booktabs}
\usepackage{makecell}
\usepackage{tabularx,tabulary}
\usepackage{colortbl}
\usepackage{multirow}
\usepackage{adjustbox}

\usepackage[colorlinks,citecolor=blue,linkcolor=red!70!black]{hyperref}
\usepackage[nameinlink]{cleveref}

\definecolor{marine}{HTML}{2c61b5}

\Crefname{EXA}{Example}{Examples}

\theoremstyle{plain}
\newtheorem{THM}{Theorem}[section]
\newtheorem{PROP}[THM]{Proposition}

\newtheorem{COR}[THM]{Corollary}

\theoremstyle{definition}
\newtheorem{DEF}[THM]{Definition}
\newtheorem{RMK}[THM]{Remark}
\newtheorem{EXA}[THM]{Example}

\DeclareMathOperator{\id}{id}
\DeclareMathOperator{\Map}{Map}
\DeclareMathOperator{\PMap}{PMap}
\DeclareMathOperator{\Homeo}{Homeo}
\DeclareMathOperator{\Int}{Int}

\newcommand{\cA}{\mathcal{A}}
\newcommand{\cM}{\mathcal{M}}
\newcommand{\cP}{\mathcal{P}}
\newcommand{\cV}{\mathcal{V}}

\newcommand{\N}{\mathbb{N}}
\newcommand{\Z}{\mathbb{Z}}

\usepackage[T1]{fontenc}
\usepackage[utf8]{inputenc}

\title[Coarsely bounded generating sets for big mapping class groups]{Coarsely bounded generating sets for mapping class groups of infinite-type surfaces}
\author{Thomas Hill, Sanghoon Kwak, and Rebecca Rechkin}
\date{}

\begin{document}

\maketitle
\begin{abstract}
Mann and Rafi's seminal work \cite{mann2023large} initiated the study of the coarse geometry of big mapping class groups. 
Specifically, they construct coarsely bounded (CB) generating sets for mapping class groups
of a large class of infinite-type surfaces.  In this expository note, we illustrate examples of surfaces whose mapping class groups admit such generating sets, as well as those that do not, with the goal of exploring the context of Mann--Rafi's hypotheses.  
\end{abstract}

\section{Introduction}

Interest in big mapping class groups has grown significantly in recent years.  
Mann--Rafi's work \cite{mann2023large}, based on Rosendal's framework \cite{rosendal2021coarse} of \emph{coarsely bounded} (CB) sets, takes the first steps in understanding the large-scale geometry of big mapping class groups, primarily for surfaces with tame end spaces.
In the specific scenario when a surface has a unique maximal end Jiménez-Rolland--Morales \cite{rolland2023large} provide a complete CB classification without tameness assumptions. The first author of this paper modified \cite{mann2023large} to establish a full CB classification of pure mapping class groups of infinite-type surfaces, similarly without tameness assumptions \cite{hill2023largescale}.  Furthermore, the influence of Mann and Rafi’s approach extends beyond big mapping class groups, offering a model for studying the large-scale geometry of other ``big'' groups as well; see, for instance, \cite{Domat_2023, domat2023generatingsetsalgebraicproperties, branman2024graphical, udall2024spherecomplexlocallyfinite}.

This expository paper aims to further illustrate the results of Mann and Rafi by offering explicit examples of CB generating sets for specific surfaces.
In \Cref{sec:prelim}, we expand upon their descriptions and examples and provide definitions related to big mapping class groups,
coarse geometry of topological groups.
\Cref{sec:CBgenset} describes the CB-generating set for $\Map(S)$
constructed by Mann--Rafi(\Cref{thm:CBgen}) and discusses why this set is indeed CB. Finally, in
\Cref{sec:examples}, we consider various examples of this CB generating set, and build non-examples with the aim of illustrating the need for Mann--Rafi's various hypotheses.  

\section*{Acknowledgments}
We thank Priyam Patel for a careful reading and useful discussions. We also thank Brian Udall for pointing out an error in \Cref{exa:nontame_CB} and suggesting a fix, and Chaitanya Tappu for helpful suggestions to improve \Cref{partialorder}.  We thank the anonymous referees for helpful comments that greatly improved this article. 

The first and third authors acknowledge support from RTG DMS--1840190 and NSF DMS--2046889.
The second author was partially supported by the New Faculty Startup Fund from Seoul National University, a KIAS Individual Grant (HP098501) via the June E Huh Center for Mathematical Challenges at Korea Institute for Advanced Study, and NSF DMS--2304774.

\section{Preliminary}\label{sec:prelim}

\subsection{Infinite-type surfaces}

A surface is a 2-manifold that is second countable, connected, orientable, and has a compact (possibly empty) boundary. In this article, we consider only surfaces without boundary. A surface is said to be of \textbf{finite type} if its fundamental group is finitely generated; otherwise, it is of \textbf{infinite type}.
Connected finite-type surfaces are classified up to homeomorphism by three integers: the number of boundary components $b$, the genus $g$, and the number of punctures $n$.
In contrast, these three invariants are not sufficient to topologically classify infinite-type surfaces. For infinite type surfaces, one also needs to consider the surface’s end space, which can be thought of intuitively as the different ways of ``going to infinity'' on the surface. However, defined this way, the structure of the end space is not always apparent from a particular drawing of the surface; for example see \Cref{fig:lochandlad}.  
Formally, the end space is defined as follows. 

\begin{figure}[ht!]
    \centering
    \includegraphics[width=.7\textwidth]{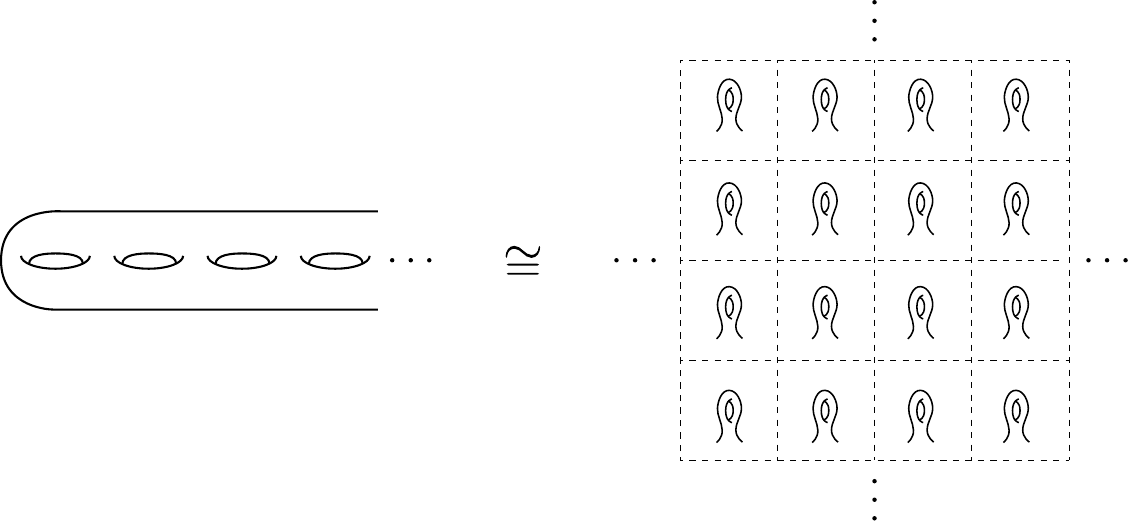}
    \caption{These two surfaces are homeomorphic and are referred to as the Loch Ness monster surface.  On the left, it is clear there is only one way to approach infinity in the surface. However, on the right, the apparently ``different ways'' to move toward infinity are identified in the one-point compactification of $\mathbb{R}^2$.}
    \label{fig:lochandlad}
    \end{figure}

\begin{DEF}

    A \textbf{compact exhaustion} $\{K_i\}_{i \ge 1}$ of a surface $S$ is a
    collection of compact subsurfaces $K_i$ such that $K_i \subset K_{i+1}$ for each $i \ge 1$, and $\bigcup_{i =1}^\infty K_i = S$.

  Fix a compact exhaustion $\{K_i\}_{i=1}^\infty$ of an infinite-type surface $S$.  Using the inclusion-induced maps $\pi_0(S\setminus K_{i+1}) \to \pi_0(S \setminus K_{i})$, the \textbf{end space} of $S$ is the inverse limit $E(S) := \displaystyle \varprojlim\pi_0(S \setminus K_i)$.  When the context is clear, we may simply write $E=E(S)$.  For every surface $S$, the end space $E(S)$ is independent of the choice of compact exhaustion of $S$.   
  
    An \textbf{end} of $S$ is a point in the end space. Ends of $S$ come in two types: those \textbf{accumulated by genus} (also called \textbf{non-planar ends}), and \textbf{planar ends}. Heuristically, an end is accumulated by genus if, as you move toward it, you always encounter genus; otherwise, it is planar.
    The subset of ends accumulated by genus is denoted by $E_G(S)$.
    There are other equivalent definitions of the end space. For a complete discussion, see \cite[Section 2.1]{aramayona2020big}.
\end{DEF}

The genus and end space are used to classify infinite-type surfaces up to homeomorphism as follows. 
\begin{THM}[{\cite{Richards1963OnTC,kerekjarto1923vorlesungen}}]\label{thm:RichardsClassification}
    Let $S$ and $S'$ be two surfaces of infinite type.  Then $S$ is homeomorphic to $S'$ if and only if
    \[
    (g(S), b(S), E(S), E_G(S)) \cong (g(S'), b(S'), E(S'), E_G(S'))
    \]
    By $\cong$ between the 4-tuples we mean that $S$ and $S'$ have the same number of boundary components: $b(S) = b(S')$; the same number of genus $g(S) = g(S')$, (possibly infinite); and there exists a homeomorphism of pairs
    \[
    (E(S), E_G(S)) \xrightarrow{\cong} (E(S'), E_G(S')),
    \]
    meaning there is a homeomorphism $E(S) \to E(S')$ whose restriction to $E_G(S)$ is a homeomorphism $E_G(S) \to E_G(S')$.
    
\end{THM}

The set of ends admits a nice topology that is compatible with the topology on
$S$ (see \cite[Section 2]{aramayona2020big}).
With this topology, the set $E(S)$ of ends  becomes compact, Hausdorff, and totally disconnected, so it follows that $E(S)$ is a closed subset of the Cantor set \cite[Theorem 2.1]{aramayona2020big}. 
The set of ends accumulated by genus $E_G(S)$ forms a closed subset of $E(S)$.

Introduced by Mann--Rafi \cite{mann2023large}, the following gives a partial order between ends of a surface.
\begin{DEF}[{\cite[Definition 4.1]{mann2023large}}] \label{partialorder}
    For an infinite-type surface $S$ and $x,y\in E(S)$, we say $y\preccurlyeq
    x$,  if and only if for every neighborhood $U\subset S$ of $x$ there exists
    a neighborhood $V \subset S$ of $y$ and there exists $f\in \Homeo(S)$ so that
    $f(V)\subset U$. This defines a preorder on $E(S)$. We say that two ends
    $x,y\in E(S)$ are \textbf{of the same type}, denoted by $x\sim y$, if
    $x\preccurlyeq y$ and $y\preccurlyeq x$. We write $x \prec y$ if $x
    \preccurlyeq y$ and $x \not\sim y$. If $x \not\preccurlyeq y$ and $y
    \not\preccurlyeq x$, then $x$ and $y$ are said to be \textbf{incomparable.}
    Taking $\sim$ as an equivalence relation on $E(S)$, for $x \in E(S)$ we
    denote by $E(x)$ the equivalence class of $x$. Then $\preccurlyeq$ defines a
    \textbf{partial order} on the set of equivalence classes of $E(S)$. We will
    simply denote by $(E,\preccurlyeq)$ the partially ordered set.
\end{DEF} 
 We have found it helpful
    to think of $y \preccurlyeq x$ as ``$y$ is at most as complicated as
    $x$.'' The following result shows that every surface admits an end of maximal complexity.

\begin{PROP}[{\cite[Proposition 4.7]{mann2023large}}]
For an end space $E$ of an infinite type surface, the partially ordered set $(E, \preccurlyeq)$ admits a maximal element.
\end{PROP}
We denote by $\mathcal{M}(E)$ the (nonempty) \textbf{set of maximal elements of $E$}
under $\preccurlyeq$.

\begin{EXA}\label{exa:partialorder}
    Consider the surface $S$ pictured in \Cref{fig:partialorder}.  
    We see $x \preccurlyeq y$ since every neighborhood of $y$ contains a homeomorphic copy of $x$, but the opposite is not true, so we have the strict order $x \prec y$.  The ends $y$ and $y'$ are of the same type: $y \sim y'$.  Similarly, we also have $y \prec z$, so transitively $x \prec z$.  The end $v$ is not comparable with any other ends of the surface $S$ since all other ends are planar.  Hence, the set of maximal ends of $S$ is  $\{v,z\}$.
    
    \begin{figure}[ht!]
    \centering
    \vspace{.5cm}
    \begin{overpic}[width=.7\textwidth]{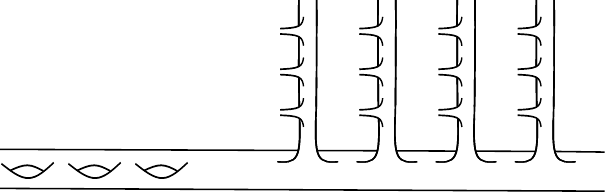}
        \put(-9,3){$v \cdots$}
        \put(42,12){$x\cdot$}
        \put(50,37){$y$}
        \put(63.5,37){$y'$}
        \put(50.2,31){$\vdots$}
        \put(63.5,31){$\vdots$}
        \put(100,2.5){$\cdots z$}
    \end{overpic}

    \caption{A surface $S$ with four end types.  Here, $x \prec y \sim y' \prec z$, and $v$ is incomparable with all other ends of $S$.}
    \label{fig:partialorder}
    
\end{figure} 

\end{EXA}

The set of maximal ends plays an important role in
several classification theorems; for example, see \Cref{thm:locCB}.

\subsection{Coarse geometry of topological groups}  
Recall that the \textbf{mapping class group} of a surface $S$, denoted by $\Map(S)$, is the group of orientation-preserving homeomorphisms of $S$ up to isotopy.   With the compact-open topology, $\Map(S)$ is a topological group.  
Finite-type surfaces have a finitely generated mapping class group by the Dehn--Lickorish theorem \cite{dehn1938gruppe,lickorish1964finite}. The word metric with respect to this generating set gives $\Map(S)$ a well-defined quasi-isometry type; that is, the (possibly distinct) word metrics induced by different finite generating sets are quasi-isometric to one another. 

On the other hand, mapping class groups of infinite-type surfaces are not even compactly generated (\cite[Theorem 4.2]{aramayona2020big}), so their large-scale geometry is more difficult to understand.
Fortunately, since mapping class groups of infinite-type surfaces are Polish (separable and completely metrizable), 
recent work of Rosendal \cite{rosendal2013global,rosendal2014largescalegeometrymetrisable, rosendal2021coarse} extends the framework of geometric group theory to understand these groups by replacing compact subsets with ``coarsely bounded'' (abbreviated CB) sets.  
In particular, Rosendal proves that if a coarsely bounded set generates a Polish group, then the group has a well-defined quasi-isometry type \cite[Proposition 2, Theorem 3]{rosendal2014largescalegeometrymetrisable}.

\begin{DEF}[{\cite[Definition 1.9]{rosendal2013global}}]
\label{def:CB}
    A subset $A$ of a topological group $G$ is globally coarsely bounded in $G$, or simply \textbf{CB}, if for every continuous action of $G$ on a metric space $X$ by isometries, the orbit $A \cdot x$ is bounded for all $x \in X$.  
    A group $G$ is said to be \textbf{locally CB} if it admits a CB neighborhood of the identity, and \textbf{CB-generated} if a CB set generates it.  
\end{DEF}

There are several equivalent definitions of CB. For example, $A \subset G$ is CB if every compatible left-invariant metric on $G$ gives $A$ finite diameter, one given in Mann--Rafi {\cite[Definition 1.1]{mann2023large}}. This definition is equivalent to \Cref{def:CB} by {\cite[Theorem 1.10]{rosendal2013global}}.

\begin{PROP} \label{prop:CBimplications} 
The implications between the CB, locally CB, and CB-generated Polish groups are summarized in the diagram below: 
\begin{center}
    \begin{tikzcd}
        \operatorname{CB} \arrow[rd, Rightarrow,"(3)",swap] \arrow[r, Rightarrow,"(1)"] & \operatorname{CB-generated} \arrow[d, Rightarrow,"(2)"] \\
            & \text{\rm locally CB}        
    \end{tikzcd}
\end{center}
\end{PROP}

\begin{proof} 
When a group is CB, it is CB-generated since the whole group can be taken as the CB-generating set, so implication (1) holds.
Implication (2) is a result of Rosendal for Polish groups \cite[Theorem 1.2]{rosendal2021coarse}. 
Of course, combining (1) and (2) results in implication (3), but (3) can also be seen directly.  If a group is globally CB, then it is locally CB since the whole
group can be taken as the CB neighborhood of the identity. 
\end{proof}

\subsection{Mann--Rafi's CB classification} 
In this section, we summarize the results of Mann--Rafi \cite{mann2023large} classifying surfaces whose mapping class groups are CB, locally CB, or CB-generated. 

We first define the two obstructions to $S$ having a CB-generated mapping class group: \emph{limit type}, and \emph{infinite rank}.
\begin{DEF}[{\cite[Definition 6.2]{mann2023large}}]
    We say that the end space $E$ is of \textbf{limit type} if 
    \begin{itemize}
        \item[(a)] there is a finite index subgroup $G$ of $\Map(S)$;
        \item[(b)] a $G$-invariant set $X \subset E$;
        \item[(c)] pairwise nonequivalent ends $z_n \in E$, indexed by $n \in \N$;
        \item[(d)] and a family of nested neighborhoods $\{U_n\}$
    \end{itemize} such that 
    \begin{align*}
    &\bigcap_n U_n = X, \qquad 
    E(z_n) \cap U_n \neq \emptyset, \qquad E(z_n) \cap U_0^c \neq \emptyset, \qquad \text{and}\\
    &E(z_n) \subset (U_n \cup (E-U_0)).
    \end{align*}
    This definition is illustrated in \Cref{fig:limit_type}.
\end{DEF}

\begin{figure}[ht!]
    \centering
    \includegraphics[width=.7\textwidth]{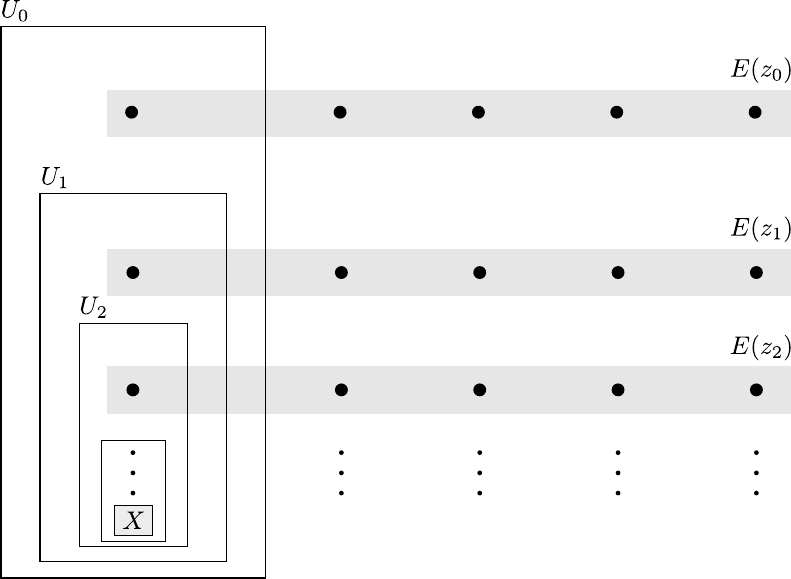}
    \caption{Illustration of an end space of limit type. Notice the elements of $E(z_n)$ overlaps $U_n$, but are not entirely contained in $U_n$, while $E(z_n) \cap U_n$'s limit into $X$.}
    \label{fig:limit_type}
\end{figure}

\begin{DEF}[{\cite[Definition 6.5]{mann2023large}}]\label{def:infiniterank}
    We say $\Map(S)$ has \textbf{infinite rank} if 
    \begin{enumerate}[label=(\arabic*)]
        \item there is a finite index subgroup $G$ of $\Map(S)$,
        \item  a closed $G$-invariant set $X \subset E$,
        \item a neighborhood $U$ of $X$ and pairwise nonequivalent stable ends $\{z_n\}_{n \in \N}$ such that
         \begin{enumerate}[label=(\alph*)]
        \item For each $n \in \N$, $E(z_n)$ is countably infinite and has at least one accumulation point in both $X$ and in $E - U$, and
        \item the set of accumulation points of $E(z_n)$ in $U$ is a subset of $X$.
        \end{enumerate} 
    \end{enumerate} 
   We illustrate this definition in \Cref{fig:infinite_rank}.  
     We say $\Map(S)$ has \textbf{finite rank} if it is not infinite rank.
\end{DEF}
 In both definitions, we often take $X$ as a finite set, and $G$ as a point-wise stabilizer of $X$ in $\Map(S)$.

\begin{figure}[ht!]
    \centering
    \includegraphics[width=.8\textwidth]{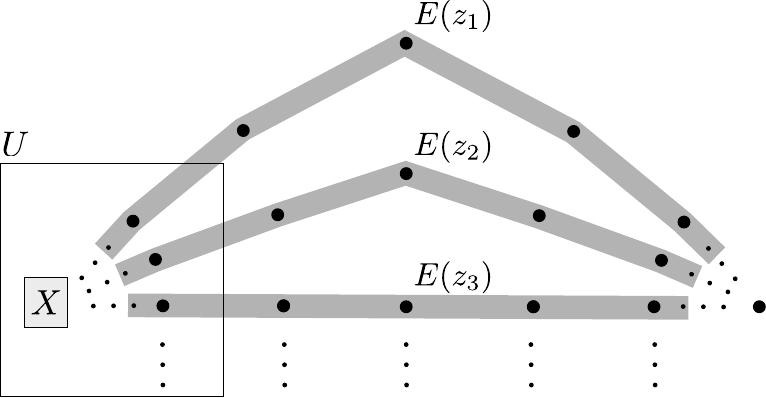}
    \caption{Illustration of an end space of a surface with infinite rank mapping class group. Notice $E(z_n)$ has at least two accumulation points, one of which is contained in $U$ and another of which is not.}
    \label{fig:infinite_rank}
\end{figure}

For surfaces with a \emph{tame} end space or surfaces with countable end
  space, Mann--Rafi completely classify when the mapping class group is CB as well as when it is CB-generated.  The partial order described above and the notion of \emph{stable ends} play a role in the definition of tameness.  We recall these definitions below.

\begin{DEF}[{\cite[Definition 4.14]{mann2023large}}]
    Let $x\in E(S)$ and let $U \subset S$ be a neighborhood of $x$.  The neighborhood $U$ is \textbf{stable} if and only if for any smaller neighborhood $U'\subset U$ of $x$, there is a homeomorphic copy $U''$ of $U$ contained in $U'$.  An end $x$ is called \textbf{stable} if it has a stable neighborhood.
\end{DEF}

Put another way, stable ends are ``locally self-similar''.

\begin{DEF}[{\cite[Definition 6.11]{mann2023large}}]
    The end space $E(S)$ (or the surface $S$) is \textbf{tame} if and only if the following two hold:
    \begin{enumerate}
        \item every maximal (with respect to the partial order in \Cref{partialorder}) end is stable and 
        \item every immediate predecessor to a maximal end is also stable.
    \end{enumerate}
\end{DEF}
We remark their description of maximal ends of $W_{A,B}$ in \cite[Definition 6.11]{mann2023large} is equivalent to the immediate predecessors of (possibly equal) maximal ends $A,B$.

Many standard examples of infinite-type surfaces have tame end spaces. For instance, every end of the Loch Ness monster surface, the Ladder surface, and the (blooming) Cantor tree surface is both maximal and stable; thus, each of these surfaces is tame. The bi-infinite flute surface is a slightly more nuanced example: its countable collection of isolated punctures are immediate predecessors of two maximal ends, which are those accumulated by punctures. Since both types of ends are stable, the bi-infinite flute surface is also tame.
Constructing a surface with a non-tame end space requires more creativity.  Mann and Rafi give one such example, which we now describe.

\begin{EXA}[Non-tame end space; {\cite[Example 6.13]{mann2023large}, \cite[Countable nonplanar case]{mann2022results}}] \label{exa:nontame_basic}

Recall that when a surface has a countable non-empty end space, then there is a countable ordinal $\alpha$ for which the end space is homeomorphic to the ordinal $\omega^\alpha n + 1$, where $\omega$ is the first infinite ordinal, and $n$ is a natural number \cite{mazurkiewicz1920contribution}. Then any end $x$ is locally homeomorphic to  $\omega^\beta + 1$, where $\beta$ is the Cantor--Bendixon rank of $x$.
 
 Let $S$ be a surface with countable end space $D = E(S)$ constructed as follows.  
 Let $D$ be a countable collection of ends such that

 \begin{equation*}
     D \cong \bigsqcup_{n \in \mathbb{N}} \left(\omega^n + 1, \{\omega^n\}\right)
 \end{equation*}
 where, by the pair $\left(\omega^n + 1, \{\omega^n\}\right)$, we mean the total end space of some subsurface $S'$ is homeomorphic to $\omega^n + 1$, and the unique maximal point is accumulated by genus. Let $z_n$ denote this maximal end accumulated by genus.  
 As $D$ is a clopen subset of a Cantor set, it is compact.  Therefore, each $\left(\omega^n + 1, \{\omega^n\}\right)$ in $D$ lies in a bounded set.  We arrange these bounded sets to Hausdorff converge to a single point, say $z$ (see \Cref{fig:nontame_nonlocCB}). Note that $z$ is a maximal end.
 
  The $\{ z_n \}_{n \ge 1}$ are pairwise incomparable because each neighborhood of $z_n$
  is homeomorphic to $(\omega^n + 1, \{\omega^n\})$, which does not contain a
  homeomorphic copy of $(\omega^m + 1, \{\omega^m\})$ for $m \ne n$. Due to
  the incomparability of these ends, the homeomorphism type of smaller and
  smaller neighborhoods of $z$ does not eventually stabilize.  This means that
  the end $z$ is not stable, and $D$ is not tame. 
  
    \begin{figure}[ht!]
        \centering
        \includegraphics[width=0.8\textwidth]{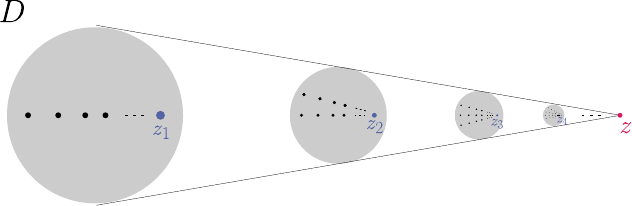}
        \caption{An illustration of the end space $D$ for \Cref{exa:nontame_basic}. The maximal ends are $z$ and $\{z_n\}_{n \ge 1}$, which are pairwise incomparable to one another. $D$ is not tame as the end $z$ is not stable.}
        \label{fig:nontame_nonlocCB}
    \end{figure}
    \end{EXA}

Now under the tameness condition, we state the Mann--Rafi's complete classifications of surfaces with CB, or CB-generated mapping class groups.

\begin{THM}[{\cite[Theorem 1.6]{mann2023large}}] \label{thm:CBgen}
    If $\Map(S)$ is CB-generated, then $E(S)$ is not of limit type, and $\Map(S)$ is of finite rank.  Moreover, for a surface with tame end space and locally CB (but not globally CB) $\Map(S)$, the converse holds.
\end{THM}

\begin{THM}[{\cite[Theorem 1.7]{mann2023large}}]
\label{thm:CB}
    If the genus of $S$ is $0$ or $\infty$, and $E(S)$ is self-similar or telescoping, then $\Map(S)$ is CB.  Furthermore, when $E(S)$ is countable or tame, then the converse holds. 
\end{THM}

The end space of a surface is \textbf{self-similar} if, for every partition $E(S) = E_1 \sqcup \dots \sqcup E_n$, there is a homeomorphic copy of the end space $D \cong E(S)$ contained in one of the $E_i$.  The telescoping condition is a generalization of self-similarity (see \cite[Definition 3.3]{mann2023large}). 

Mann--Rafi also give a simpler obstruction for a surface to have a CB mapping class group, for which we need the following definition:

\begin{DEF}[{\cite[Definition 1.8]{mann2023large}}]
    We say two subsurfaces $R$ and $R'$ of $S$ \textbf{intersect} if every subsurface homotopic to $R$ intersects every subsurface homotopic to $R'$.  A connected finite-type subsurface $R \subset S$ is \textbf{non-displaceable} if $R$ and $f(R)$ intersect for every $f \in \Homeo(S)$.
\end{DEF}

\begin{THM}[{\cite[Theorem 1.9]{mann2023large}}] \label{thm:nondisplaceable}
    If $S$ contains a non-displaceable subsurface (of finite type), then $\Map(S)$ is not CB.
\end{THM} 

The presence of non-displaceable subsurfaces greatly reduces the number of surfaces with CB \emph{pure} mapping class group, i.e., the subgroup of the mapping class group fixing the ends pointwise.  In particular: 
\begin{THM}[{\cite[Proposition 1.2]{hill2023largescale}}] \label{thm:nondisplaceable}
    If $|E(S)| \ge 3$ (possibly infinite), then $S$ contains a subsurface that is non-displaceable by pure mapping classes.  
\end{THM} 

Consequently, the only surface with a globally CB pure mapping class group is the Loch Ness monster surface \cite[Theorem 1.1(a)]{hill2023largescale}.

In the case when $S$ has a unique maximal end, the surfaces admitting a globally CB mapping class group is greatly reduced by the work of Jim\'enez-Rolland--Morales: 

\begin{THM}[{\cite[Theorem 1.3]{rolland2023large}}]
Let $S$ be an infinite-type surface \emph{with a unique maximal end} and suppose that $\Map(S)$ is locally CB.  Then $\Map(S)$ is globally CB if and only if the genus of $S$ is 0 or $\infty$.  
    
\end{THM}

In the context of \cite{rolland2023large} and \cite{hill2023largescale}, the conditions of being CB-generated and locally CB coincide.  This is not true in general, regardless of the tameness of the end space (see \Cref{exa:kanagawa,exa:nontame_locCB}). However, Mann--Rafi
completely classified the surfaces with locally CB mapping class groups regardless of whether the end space of the surface is tame or not. 

\begin{THM}[{\cite[Theorem 5.7]{mann2023large}}] \label{thm:locCB}
    $\Map(S)$ is locally CB if and only if there is a finite-type surface (not necessarily compact) $K \subset S$ satisfying the following:
    \begin{enumerate}[label=(\arabic*)]
    \item the genus of each component of $S \setminus K$ is 0 or $\infty$, and
    \item $K$ partitions $E=E(S \setminus \operatorname{int}(K))$ into finitely many clopen sets 
    \[
        E = \left( \bigsqcup_{A \in \cA} A\right) \sqcup \left( \bigsqcup_{P \in \cP} P\right),
    \]
    such that
    \begin{enumerate}[label=(\alph*)]
        \item Each $A \in \cA$ is self-similar, $\cM(A) \subset \cM(E)$, and $\cM(E) = \bigsqcup_{A \in \cA}\cM(A)$.
        \item Each $P \in \cP$ is homeomorphic to a clopen subset of some $A \in \cA$. 
        \item For any $x_A \in \cM(A)$, and any neighborhood $V$ of the end $x_A$ in $S$, there is $f_V \in \Homeo(S)$ so that $f_V(V)$ contains the complementary region to $K$ with end set $A$.
        \end{enumerate}
    \end{enumerate}
    Moreover, in this case, $\cV_K = \{g \in \Homeo(S): g|_K = \id\}$ is a CB neighborhood of the identity, which witnesses that $\Map(S)$ is locally CB.  Also, we may always take $K$ to have no genus if $S$ has infinite genus, and genus equal to that of $S$ otherwise.  If the number of isolated planar ends(equivalently, the punctures) of $S$ is finite, we may choose $K$ to have all of the isolated planar ends so that $S \setminus K$ has none.
\end{THM} 

We will only consider surfaces that satisfy \Cref{thm:locCB} with $\mathcal{P} = \emptyset$.  This need not always be the case; for an example with $\mathcal{P} \ne \emptyset$, see \cite[Example 5.10]{mann2023large}.

\begin{RMK}
    We will often refer to condition 2(c) of \Cref{thm:locCB} as the ``small zoom condition.'' Intuitively, this condition says that if we examine a maximal end $x_A$ in $\bigcup_{A \in \cA} A$ closely, we can always zoom in further to find a homeomorphic copy of $A$ contained within it.
\end{RMK}

Note that our statement of \Cref{thm:locCB} differs slightly from the wording {\cite[Theorem 5.7]{mann2023large}}. In their statement, Mann--Rafi say that $K$ partitions $E(S)$, whereas we state that it partitions $E(S \setminus \Int(K))$. However, it is clear from the final paragraph of their theorem statement, as well as from the proof, that they intend $E(S \setminus\Int(K))$.

\section{CB-generating set for $\Map(S)$}
\label{sec:CBgenset}

 Let $S$ be a surface whose mapping class group is CB-generated as a consequence of \Cref{thm:CBgen}, i.e., a surface with tame end space, not of limit type, with finite rank mapping class group.  The goal of this section is to describe the CB-generating set defined in the proof of \cite[Theorem 1.6]{mann2023large}. These include mapping classes that arise in the setting of finite-type surfaces, as well as shift maps which are unique to infinite-type surfaces. Such maps were first considered in \cite[Section 6]{patel2018algebraic}, and are more generally defined as follows.

\begin{DEF}
    Consider the infinite strip $\mathbb{R} \times [-1,1]$, and remove infinitely many open disks of radius $\frac14$
    centered at the points $(n, 0)$ for $n \in \mathbb{Z}$.  Now, fix a single
    surface $\Sigma$ that has exactly one boundary component.  At each of the
    removed disks of the strip, attach a copy of $\Sigma$ by gluing its boundary
    along the boundary of the removed disk.  Call the resulting surface
    $\Sigma'$.

    If a surface $S$ has a subsurface homeomorphic to $\Sigma'$ as constructed above, then there is a $\textbf{generalized shift}$, or more simply a \textbf{shift}, $h \in \Map(S)$ which acts like a translation within the strip.  Specifically, within the strip $\Sigma'$,  $h$ sends $(x, y) \mapsto (x + 1, y)$ outside an $\varepsilon$ neighborhood of the boundary of $\Sigma'$, tapering to the identity on $\partial \Sigma'$, and extending by the identity on the rest of $S$.   
    When the surface attached at each removed disk of the strip is a torus with one boundary component, the surface $\Sigma'$ is an infinite genus strip and the shift map $h$ is called a \textbf{handle shift}, see \Cref{fig:handleshift}.  If $\Sigma$ is a punctured disk, then $h$ is called a \textbf{puncture shift}. In general, we will call the generalized shift $h$ a \textbf{$\Sigma$-shift}.
\end{DEF}

\begin{figure}[ht!]
    \centering
    \includegraphics[width=1\textwidth]{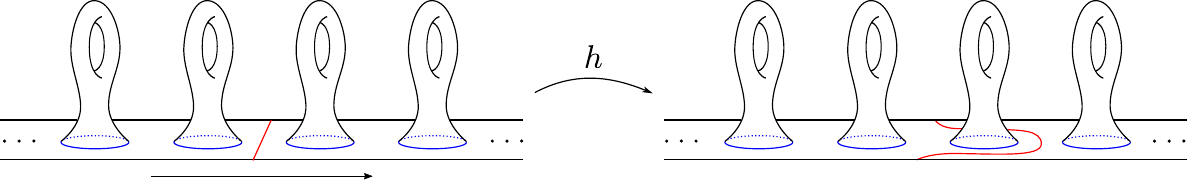}
    \caption{A handle shift map $h$.}
    \label{fig:handleshift}
\end{figure}

A pair of isolated punctures in a surface can be permuted via a \emph{half-twist} defined as follows. 

\begin{DEF}

   Let $S$ be an infinite-type surface with two punctures $a$ and $b$ (possibly with many other ends).
   \begin{figure}[ht!]
    \centering
    \begin{overpic}[width=.8\textwidth]{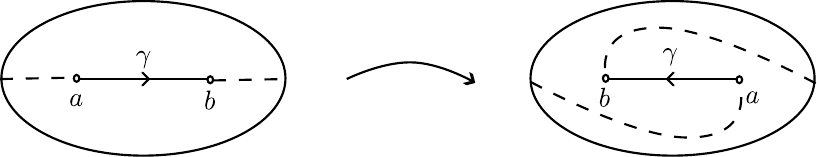}
    \put(30,17){$\alpha$}
    \put(95,17){$\alpha$}
    \end{overpic}
    \caption{Permuting two ends $a$ and $b$ of the same type via a half-twist. The arc $\gamma$ is a path from $a$ to $b$ to illustrate what happens to a small disk neighborhood of $a$ and $b$.}
    \label{fig:halftwist}
\end{figure}
   Let $\alpha$ be a separating simple closed curve such that the cut surface has two connected components, one of which is homeomorphic to a disk with the two punctures $a$ and $b$.  Let $\gamma$ denote an arc in this disk from $a$ to $b$. A \textbf{half-twist} is a simultaneous point push of $a$ along $\gamma$ to $b$ and point push of $b$ along $\gamma^{-1}$ to $a$ as depicted in \Cref{fig:halftwist} (see also \cite[Section 9.1.3]{farb2011primer}).

   Similarly, any two ends of the same type, say $x$ and $y$, can be permuted using a half-twist: simply replace the punctures $a$ and $b$ in \Cref{fig:halftwist} with a small neighborhood of $x$ and $y$.

\end{DEF}

The proof of \cite[Theorem 1.6]{mann2023large} is constructive in the sense that Mann--Rafi show that the union of the following is a CB-generating set\footnote{Mann--Rafi also show that the first four sets together generate a group containing $\PMap(S)$ (see \cite[Lemma 6.21]{mann2023large}). By  \cite[proof of Theorem 1.1(c)]{hill2023largescale}, $\cV_K\cup D\cup H$ suffice to generate $\PMap(S)$.} for $\Map(S)$:
\begin{itemize}
    \item the identity neighborhood $\cV_K$,
    \item a finite generating set $D$ for $\Map(K)$, 
    \item a finite collection $F$ of primitive \emph{generalized shifts},
    \item a finite collection $H$ of \emph{handle shifts}, and
    \item a finite collection of half-twists $T$ permuting some 
      of the ends (see below).
\end{itemize}

Since CB-generated groups are necessarily locally CB (\Cref{prop:CBimplications}), it follows from \Cref{thm:locCB} that $\Map(S)$ admits a CB identity neighborhood $\cV_K$. Moreover, because the union of a CB set with finitely many additional elements remains CB, the full generating set described above is indeed CB.

We now explain why only finitely many elements of each generator type are needed. First, since $K$ is a finite-type subsurface, $\Map(K)$ is finitely generated by a set $D$ by the classical Dehn–Lickorish theorem \cite{dehn1938gruppe, lickorish1964finite}. Any mapping class supported entirely within a single connected component of $S \setminus K$ can be generated using elements from $\cV_K \cup D$. Similarly, compositions of maps supported on distinct components can be generated by elements of $\cV_K \cup D$. However, more general elements of $\Map(S)$ may move data between different components of $S \setminus K$. To generate all such maps, it is enough to include the following additional generators: the handle shifts $H$, the generalized shifts $F$, and the half-twists $T$.

To see why $F$ and $H$ can be taken to be finite sets, first observe that $K$ partitions $E(S)$ into finitely many connected components. Let $S_i$ and $S_j$ be two distinct components of $S \setminus K$. For any type of shift involving an end in $S_i$ and one in $S_j$, it suffices to include a single such shift in the generating set because all other shifts of the same type between $S_i$ and $S_j$ can be obtained by conjugating with elements of $\cV_K$ and $D$. Since there are only finitely many such pairs $(S_i, S_j)$, we require only finitely many shifts of each type. Moreover, if infinitely many distinct types of shifts were necessary between $S_i$ and $S_j$, then $\Map(S)$ would have infinite rank, a contradiction.

A similar argument applies to the set of half-twists $T$. Any permutation of ends within a single component of $S \setminus K$ can again be generated using $\cV_K \cup D$. For permutations involving ends in different components, it suffices to include one representative half-twist for each equivalence class of ends (since only equivalent ends may be permuted). If there were infinitely many such equivalence classes between two components, $S$ would be of limit type, again, a contradiction.

To make these arguments more concrete, we will work through several explicit examples in \Cref{sec:examples}.

\section{Examples}\label{sec:examples}

In this section, we consider examples of the CB-generating set of $\Map(S)$ for surfaces satisfying the hypothesis of \Cref{thm:CBgen}, and hence generated by the CB set described in \Cref{sec:CBgenset}. Through these examples and accompanying discussion, we also aim to illustrate why the generating set introduced by Mann--Rafi \cite{mann2023large} never requires infinitely many handle shifts, generalized shifts, or half-twists.

In \cite[proof of Theorem 13]{vlamis2019notes}, Vlamis describes an explicit CB-generating of the surfaces with finitely many ends accumulated by genus, i.e. with $p = 0$ in the notation below. This description coincides with the set given in this example without the half-twists $g'$.  As a starting point for the examples in this section, we recall and elaborate on this example now.  

\begin{EXA}[Finite end space] \label{exa:np_finite}

Let $S$ be a surface with $n$ ends accumulated by genus $\{x_1, \dots, x_n\}$, and $p$ punctures $\{y_1, \dots, y_p\}$ where $1 < n < \infty$  and $0 \le p < \infty$.  Let $K$ be a finite type subsurface containing all $p$ punctures of $S$, and with $n$ boundary components such that each connected component of $S \setminus K$ contains exactly one end accumulated by genus.  By \Cref{thm:locCB}, $\cV_K$ is CB since $S\setminus K$ partitions $E(S\setminus \Int(K))$ into a disjoint union of clopen sets $A_1,\ldots, A_n$ where each $A_i$ is the set of a single end accumulated by genus and is thus self-similar and satisfies Condition (2c) of \Cref{thm:locCB}. Let $D$ be a finite generating set for $\Map(K)$.  Note that within $D$ there are the half-twists $g'_{j, j+ 1}$ permuting the punctures $y_j$ and $y_{j + 1}$ for $j = 1, \dots p - 1$  If $p = 0$ or $1$, no $g'$ maps are included.  
Because there are only finitely many ends, all ends are isolated, and $\Map(S)$ admits no generalized shifts $F$. 
On the other hand, we now require $n -1$ pairwise commuting handle shifts $\{h_1, \dots, h_{n - 1}\}$ for which $h_i^+ = x_n$ and $h_i^- = x_i$.  
Lastly, let $g_{i, i +1}$ be half-twists that permute the ends $x_i$ and $x_{i + 1}$ accumulated by genus for $i = 1, \dots, n - 1$. This collection of half twists generates all permutations of the ends accumulated by genus.  See \Cref{fig:np_finite} for an example.  In summary, the CB-generating set for $\Map(S)$ is: 
\begin{equation*}
    \cV_K \cup D \cup \{h_1, \dots, h_{n - 1}\} \cup \{ g_{i, i + 1} \mid i = 1, \dots, n- 1\}.
\end{equation*}

\begin{figure}[ht!]
    \centering
    \includegraphics[width=.8\textwidth]{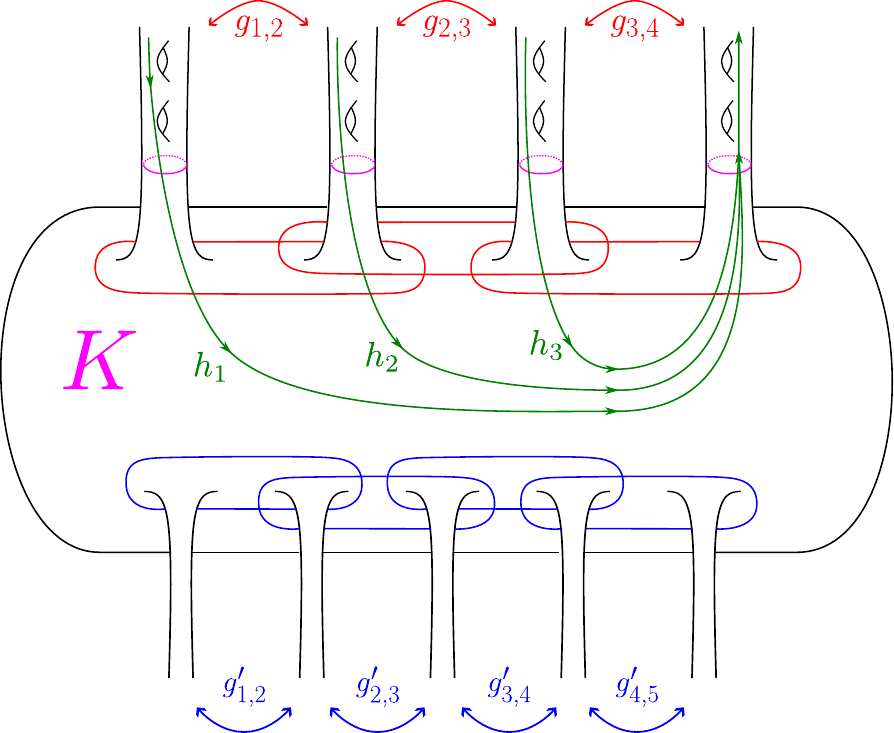}
    \caption{A surface with finitely many punctures and ends accumulated by genus illustrating \Cref{exa:np_finite}, with $n=4$ and $p=5$.  Note that the surface $K$ contains all punctures of $S$.  The maps $g'_{j, j+1}$ are contained in the finite generating set for $\Map(K)$.}
    \label{fig:np_finite}
\end{figure}

If $n = 1$, i.e., if $S$ is a punctured Loch Ness monster surface (see \Cref{fig:LNMonsterWithPunctures}), $\Map(S)$ is not locally CB and hence not CB generated.  Indeed, the small zoom condition \Cref{thm:locCB}(2c) fails.  To see why, let $V$ be a neighborhood of the end accumulated by genus.  Any homeomorphism $f \in \Map(S)$ will preserve the homeomorphism type of the disjoint subsurfaces $K$, $S \setminus (V \cup K)$, and $V$.   However, in general, $S \setminus (V \cup K)$ has non-zero genus, but $S\setminus (S \setminus K) = K$ does not.  So there is no $f \in \Map(S)$ such that $f(V) \supset S \setminus K$, and therefore $\Map(S)$ is not locally CB.  Consequently, it is not CB-generated either. 

\begin{figure}[ht!]
    \centering
    \begin{overpic}[width=5cm]{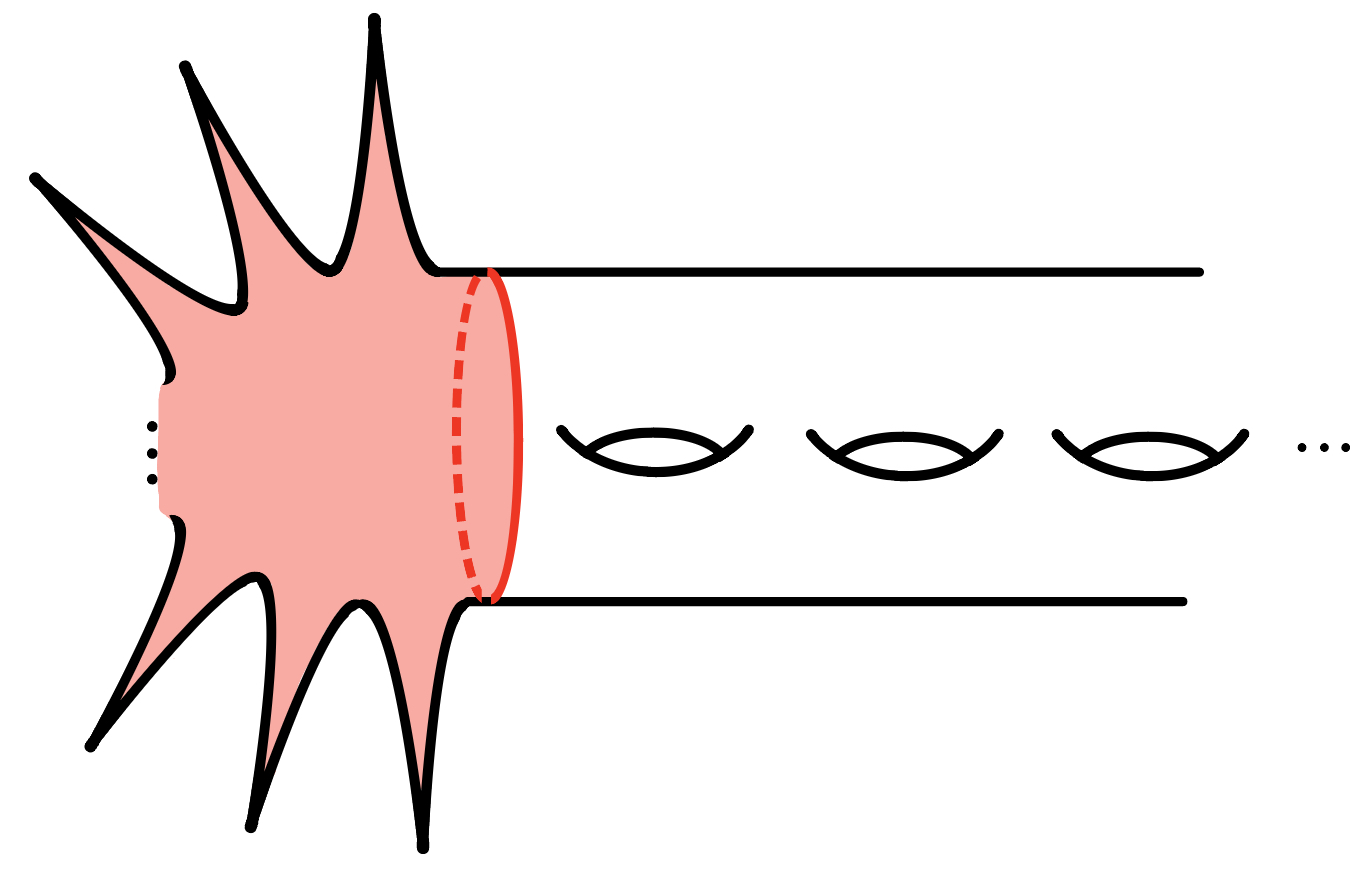}
    \put(20,30){\color{red} $K$}
    \end{overpic}
    \caption{The subsurface $K$ is non-displaceable and is also the unique subsurface satisfying the hypothesis of \Cref{thm:locCB}.}
    \label{fig:LNMonsterWithPunctures}
\end{figure}
\end{EXA}

\subsection{Connecting finiteness of $H$, $F$, and $T$ to \Cref{thm:CBgen}}

The examples in this section aim to show how attempts to modify a surface in a way that seems to require infinitely many generalized shifts result in a surface with end space of limit type or infinite rank mapping class group, i.e., failing the hypotheses of \Cref{thm:CBgen}.  

In \Cref{exa:np_finite}, the finiteness of the ends accumulated by genus implies that $H$ is finite. Similarly, since $S \setminus K$ has only finitely many complementary components each with the same end type, only finitely many half-twists are needed. For a surface with infinitely many ends accumulated by genus, this is clearly not the case. In this setting, the choice of $K$ in \Cref{thm:locCB} differs significantly, and generalized shifts become necessary. For example, consider the following modification to \Cref{exa:np_finite}, with $p = 0$ and infinitely many ends accumulated by genus:

\begin{EXA}[Chimney Surface]\label{exa:MinfPfin}
   Let $S$ be the infinite chimney surface; that is, the surface with end space $(E(S), E_G(S)) \cong (\omega^2 + 1, \omega^2 + 1)\sqcup(\omega^2 + 1, \omega^2 + 1)$, again following the notation of \Cref{exa:nontame_basic}.  A CB-generating set for this mapping class group will require a single ``chimney shift'', and the addition of handle shifts outside of $\cV_K$ is not necessary, as we explain below.  
    
    For completeness, we now describe a complete CB-generating set.  Label the maximal ends of $S$ as $e_1$ and $e_2$, and let $K$ be an annulus separating $e_1$ and $e_2$. 
    See \Cref{fig:MinfPfin}. 
\begin{figure}[ht!]
    \centering
    \includegraphics[width=.7\textwidth]{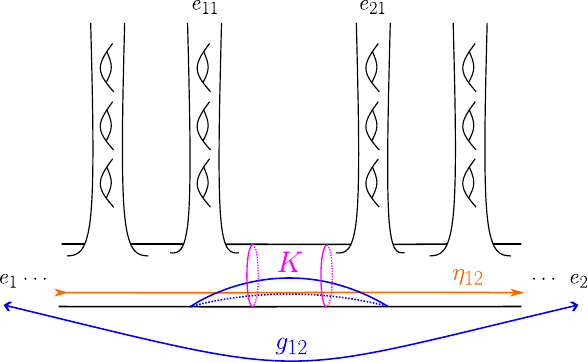}
    \caption{A surface illustrating \Cref{exa:MinfPfin}.}
    \label{fig:MinfPfin}
\end{figure}
    By \Cref{thm:locCB}, $\cV_K$ is CB. 
    Let $D$ contain the single Dehn twist generating $\Map(K)$.
    Then $F$ contains only the generalized shift $\eta_{12}$ between the two maximal ends $e_1$ and $e_2$, which shifts each end accumulated by genus to the right. 
    Note that it is not necessary to include a handle shift from $e_{11}$ to $e_{21}$ since it is conjugate (by $\eta_{12})$ to a handle shift in $\cV_K$. 
    In addition, let $g_{12}$ be a half-twist that swaps the ends $e_1$ and $e_2$. 
     
    We can permute the ends $e_{11}$ and $e_{21}$ by conjugating a permutation of ends in $\mathcal V_K$ by $\eta_{12}$. 
     In summary, the CB-generating set for $\Map(S)$ is
    \[
        \mathcal V_K\cup D \cup \{\eta_{12}, g_{12} \}.
    \]
    
\end{EXA}

\vspace{.5cm}

Increasing the complexity of the ends accumulated by genus from \Cref{exa:np_finite} to \Cref{exa:MinfPfin} in a way suggests that infinitely many handle shifts or half-twists would be required in the CB-generating set, highlighting a role of generalized shifts in the CB-generating set.   
In \Cref{exa:nmp_multgenshift} and \Cref{exa:infgenshift}, we apply a similar idea but now focus on increasing the number of higher-order shifts that appear to be required. In contrast, these modifications cause the surface to fail the hypotheses of \Cref{thm:CBgen}.
One might expect that modifying a surface in this way, so that infinitely many generalized shifts are effectively needed, would yield a mapping class group of infinite rank. While that can be the case (see \Cref{exa:kanagawa}), it is also possible that the resulting surface will have end space of limit type, as illustrated below in \Cref{exa:infgenshift}.

\begin{EXA}[Generalized shifts, $|F| = k$]
\label{exa:nmp_multgenshift}

    \begin{figure}[ht!]
    \centering
    \includegraphics[width=\textwidth]{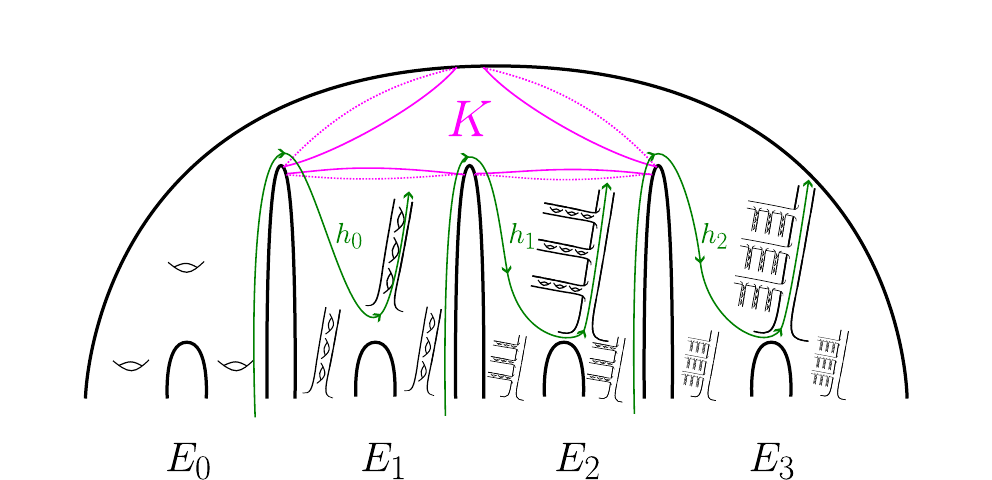}
    \caption{The surface $S$ described in \Cref{exa:nmp_multgenshift} when $k=4$.}
    \label{fig:nmp_multgenshift}
\end{figure}
 
     Consider the Cantor tree surface. Let $k$ be a positive integer and partition the end space into $k$
     disjoint subsets $E_0, E_1,E_2, \ldots, E_{k-1}$. Let $K$ be the subsurface
     with no genus and containing no ends that realizes the partition of the
     end space into the $E_i$. Call $T_0,T_1,T_2, \ldots, T_{k-1}$ the complementary components of $K$, with end space being $E_0,E_1,E_2, \ldots, E_{k-1}$ respectively. Assume $E_0$ is accumulated by genus, by periodically connect-summing handles in $T_0$ toward each end in $E_0$. Next, we assume every end of $E_1$ is accumulated by Loch Ness monster surfaces, by periodically connect-summing a sequence of Loch Ness monster surfaces in $T_1$ toward each end in $E_1$. More generally, for $i\geq 1$, call $L_i$ the infinite type surface with infinite genus, whose end space pair is $(E(S), E_G(S)) \cong (\omega^{i-1}+1,\omega^{i-1}+1)$. Then we assume every end of $E_i$ is accumulated by copies of $L_i$, by periodically connect-summing a sequence of $L_i$'s in $T_i$ to each end in $E_i$.  
     Call the resulting surface $S$, see \Cref{fig:nmp_multgenshift} for an illustration.

    We will now construct a CB-generating set for $\Map(S)$ following \Cref{sec:CBgenset}.
    Since each $E_i$ is self-similar, $\cV_K$ is CB by \Cref{thm:locCB}.
    Let $D$ be a finite collection of Dehn twists generating $\Map(K).$ 

    Next, we consider the various shift maps. 
    From the discussion in \Cref{sec:CBgenset}, we see that for each pair $E_i$ and $E_j$ with $i\neq j$ it suffices to add one shift of each possible type between some end in $E_i$ and some end in $E_j$ into our generating set. However, as in example \Cref{exa:MinfPfin}, lower-order shifts can be obtained as conjugates by elements in $\cV_K$.  Hence, we need only include a single $L_i$-shift between $E_i$ and $E_{i+1}$ for all $i \ge 1$, and for $i = 0$ a single handle shift.  Call these shifts $\{h_i\}_{i = 0}^{k-1}.$ 
    
    For each $i\neq j$, any end in $E_i$ and any end in $E_j$ are incomparable, so there are no half-twists permuting the ends between $E_i$ and $E_j$. Hence, 
    \begin{equation*}
        \cV_K \cup D \cup \{h_i\}_{i = 0}^{k - 1}
    \end{equation*}
    is a CB-generating set for $\Map(S)$.

\end{EXA}

\vspace{.5cm}

In \Cref{exa:nmp_multgenshift}, the final CB-generating set for $\Map(S)$ includes one handle shift $h_0$ and $k - 1$ generalized shifts $h_1,\ldots, h_{k-1}$. In the next example, we build on this construction by modifying the surface in a way that seems to suggest that infinitely many generalized shifts would be required in the generating set. Unlike the previous case, this surface has an end space of limit type and thus fails to have CB-generated mapping class group by \Cref{thm:CBgen}.

\begin{EXA}[Non-example: limit type and finite rank]
    \label{exa:infgenshift}
    Consider the Cantor tree surface. Partition the end space into countably many
     disjoint sets $E_0, E_1,E_2, \ldots$, each of which homeomorphic to the Cantor set. For example, when we realize the Cantor set as the set of infinite binary sequences $2^{\mathbb{N}}$, set $E_i$ as the set of sequences with exactly $i$-many $1$'s as their prefix. Note not all $E_i$ can be clopen as the end space of the whole surface is compact.
      Let $K$ be an \emph{infinite} type subsurface of genus 0, and not containing ends of $S$, that realizes this partition; i.e., $E(S\setminus \Int(K)) = \bigsqcup_i E_i$.  
     Call $T_0,T_1,T_2, \ldots$ the complementary components of $K$, with end space being $E_0,E_1,E_2, \ldots$ respectively.
     Assume $E_0$ is accumulated by genus, by periodically connect-summing handles in $T_0$ toward each end in $E_0$. Next, we assume every end of $E_1$ is accumulated by Loch Ness monster surfaces, by periodically connect-summing a sequence of Loch Ness monster surfaces in $T_1$ toward each end in $E_1$. For each $i\geq 2$, call $L_i$ the infinite type surface with infinite genus, whose end space pair is homeomorphic to $(\omega^{i-1}+1,\omega^{i-1}+1)$. Then we assume every end of $E_i$ is accumulated by copies of $L_i$, by periodically connect-summing a sequence of $L_i$'s in $T_i$ to each end in $E_i$.
     Call the resulting surface $S$ and see \Cref{fig:infax} for an illustration. 

\begin{figure}[ht!]
    \centering
    \begin{overpic}[width=\textwidth]{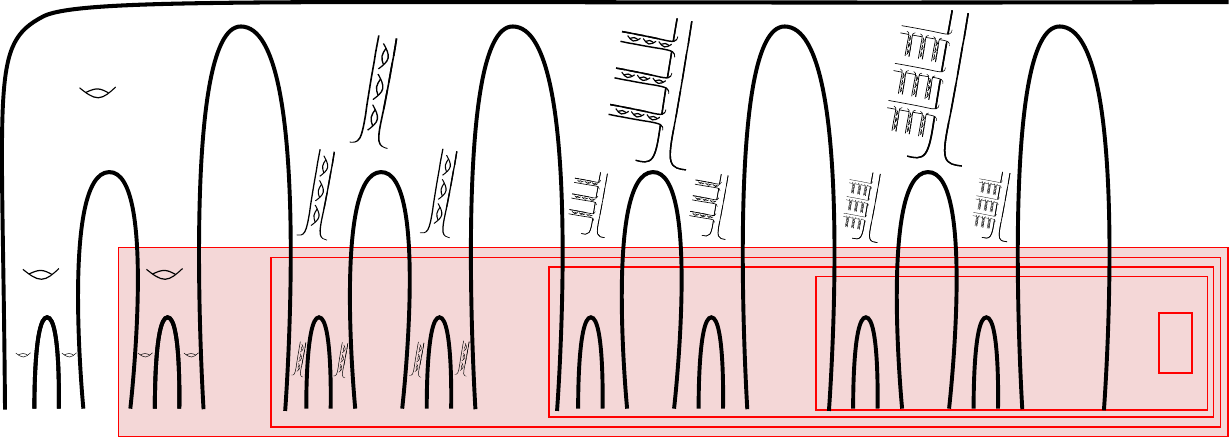}
    \put(6,0){\color{red}{$U_0$}}
    \put(18,1){\color{red}{$U_1$}}
    \put(41,2){\color{red}{$U_2$}}
    \put(63,3){\color{red}{$U_3$}}
    \put(94.4,6.5){\color{red}{$X$}}
    \put(85,7){\color{red}{$\cdots$}}
    \end{overpic}
    \caption{The surface $S$ described in \Cref{exa:infgenshift}. The neighborhoods $U_i$ illustrate that $E(S)$ is limit type.}
    \label{fig:infax}
\end{figure}

As in \Cref{exa:nmp_multgenshift}, each $E_i$ is self-similar. However, there is no finite-type subsurface $K \subset S$ that realizes the infinite partition $ E(S) = \bigsqcup_i E_i$, and so we cannot apply \Cref{thm:locCB} in the same way. In fact, the mapping class group $\Map(S)$ is not locally CB.

To see this, notice that any finite-type subsurface $K' \subset S$ such that $K'$ induces a finite clopen partition of $ E(S)$ into sets $A_1, \ldots, A_\ell$. One of these clopen sets must contain the distinguished maximal end $e$; without loss of generality, assume $e \in A_1$. Observe that $A_1$ is not self-similar and is not homeomorphic to any clopen subset of the other $A_i$. Hence, no such $K'$ satisfies the conditions of \Cref{thm:locCB}, and $\Map(S)$ is not locally CB. Consequently, $\Map(S)$ cannot be CB-generated.

Additionally, the end space of $S$ is of limit type, as illustrated by the neighborhoods $U_i$ in \Cref{fig:infax}, and the mapping class group $\Map(S)$ has finite rank since the various end types have only one accumulation point and hence fail \Cref{def:infiniterank}(3a).

\end{EXA}

\subsection{Non-tame examples}
So far, we have only considered surfaces with tame end space in this section, but we can also look at some surfaces that do not have tame end space and consider whether or not they are CB-generated. For these surfaces we have fewer tools at our disposal because we will not be able to use Mann--Rafi's CB-generating set.

\begin{EXA}[$D$ is not CB, but $D' := \sqcup_{\Z} D$ is CB]
\label{exa:nontame_CB}
  We saw that the surface $S$ with endspace $D$ in \Cref{exa:nontame_basic} is not tame.  
Still, we can apply \Cref{thm:CB} to show $\Map(S)$ is not CB, as $D$ is countable. Furthermore, we see by \Cref{thm:locCB} that $\Map(S)$ is not locally CB, which also implies $\Map(S)$ is not CB and not CB-generated by \Cref{prop:CBimplications}. 
  Because the ends $z_1,z_2,\ldots $ are all incomparable, $D$ is not self-similar. To see this, consider the partition of the end space that separates the ends of type $\omega+1$ from all the other ends in $D$ where $E_1$ contains all the ends of type $\omega+1$ and $E_2$ contains all the other ends in $D$. We would like $E_2$ to contain a homeomorphic copy of $D$.  However, $E_2$ has no ends of type $\omega+1$, so it cannot contain a homeomorphic copy of $D$.  Furthermore, there is no way for us to finitely partition $D$ in such a way that each component of the partition is self-similar, so $S$ has a non-locally CB mapping class group.  In particular, this also means that this mapping class group is not CB either.  

\begin{figure}[ht!]
        \centering
        \includegraphics[width=.6\textwidth]{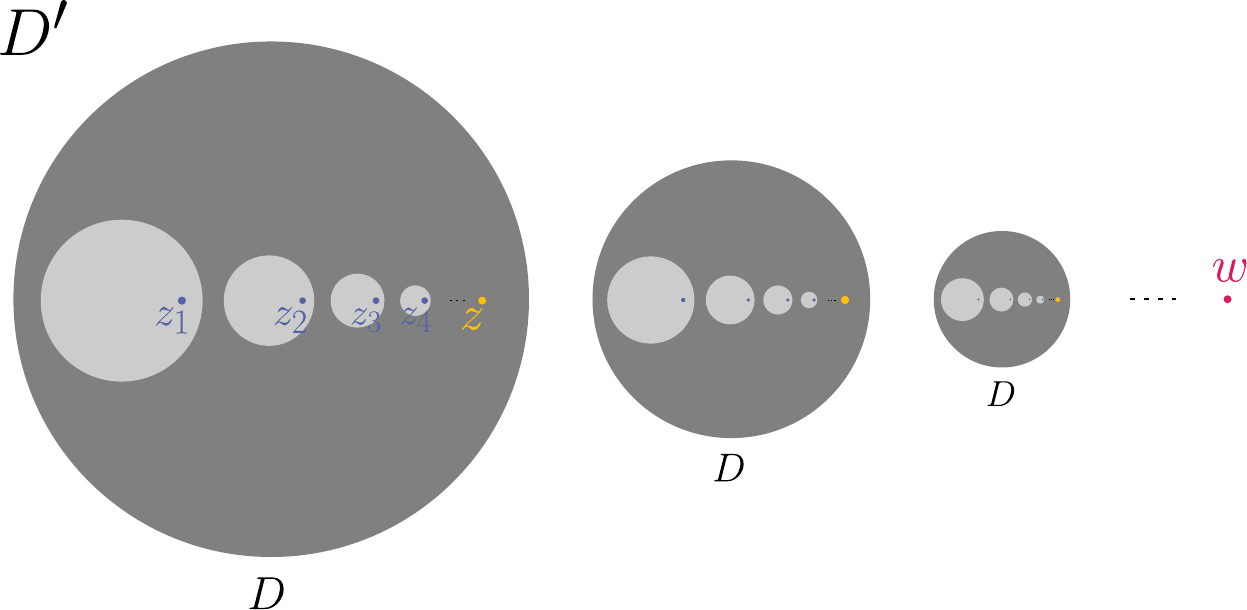}
        \caption{The end space $D'$ for \Cref{exa:nontame_CB}. The maximal end is $w$, and its immediate predecessors are the $z$'s in the copies of $D$ pictured in  \Cref{fig:nontame_nonlocCB}. 
        }
        \label{fig:nontame_CB}
    \end{figure}

Now, consider $D'=\bigsqcup_{n\in\mathbb Z}D$ arranged with a unique maximal end
$w$ as in \Cref{fig:nontame_CB}. Like $D$, $D'$ is not tame. However, $D'$ is countable and self-similar.  Indeed, given any partition of $D'$, the component of the partition containing $w$ will contain a homeomorphic copy of $D'$.  Therefore, the mapping class group of the surface with end space $D'$ is CB by \Cref{thm:CB}.

\end{EXA}

\begin{EXA}[Locally CB, non-tame, non-CB-generated]\label{exa:nontame_locCB}

    Consider a surface with end space $D'' := D' \sqcup D'$ depicted in \Cref{fig:nontame_locCB}.
    \begin{figure}[ht!]
    \centering
    \includegraphics[width=.7\textwidth]{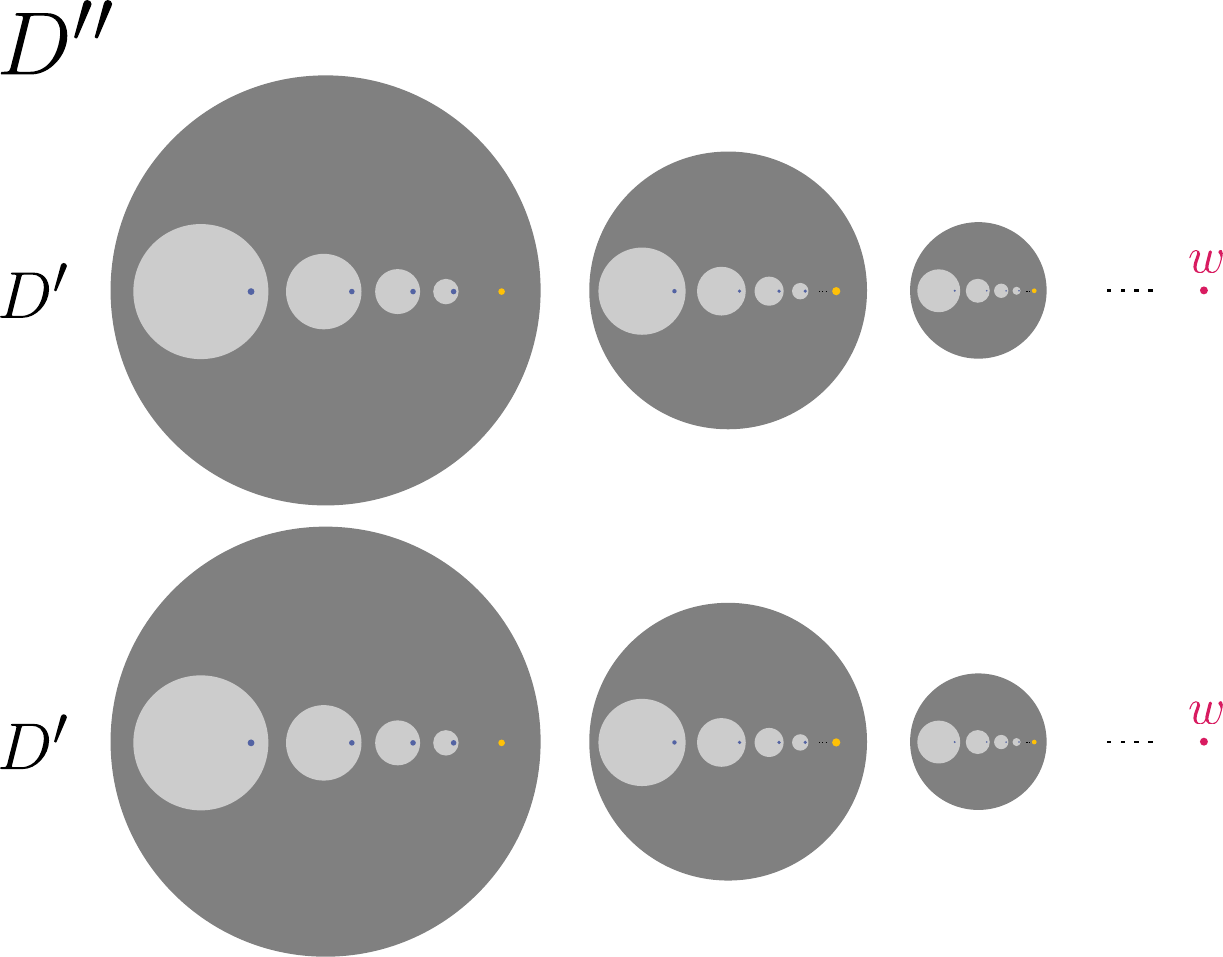}
    \caption{The end space $D'' := D' \sqcup D'$. It is still not tame as their immediate predecessors are not stable, as in \Cref{exa:nontame_CB}.  However, it follows that it realizes a surface with locally CB mapping class group, as the partition $D'' = D' \sqcup D'$ consists of only self-similar end spaces $D'$.}
    \label{fig:nontame_locCB}
\end{figure}
    Like $D'$ and $D$, $D''$ is not tame. However, unlike the mapping class group of a surface with end space $D'$, the mapping class group of a surface with end space $D''$ is locally CB, \emph{but not} CB since this countable end space is neither self-similar nor telescoping. To see local CB-ness, partition $D''$ into two components, each containing exactly one copy of $D'$.  Because the mapping class group of a surface with end space $D'$ is CB, by \Cref{thm:locCB}, the mapping class group of a surface with end space $D''$ is locally CB.  However, this mapping class group is not CB-generated since $D''$ is of limit type as is illustrated by \Cref{fig:nontame_locCB_2}. 
    
    \begin{figure}[ht!]
    \centering
    \includegraphics[width=.6\textwidth]{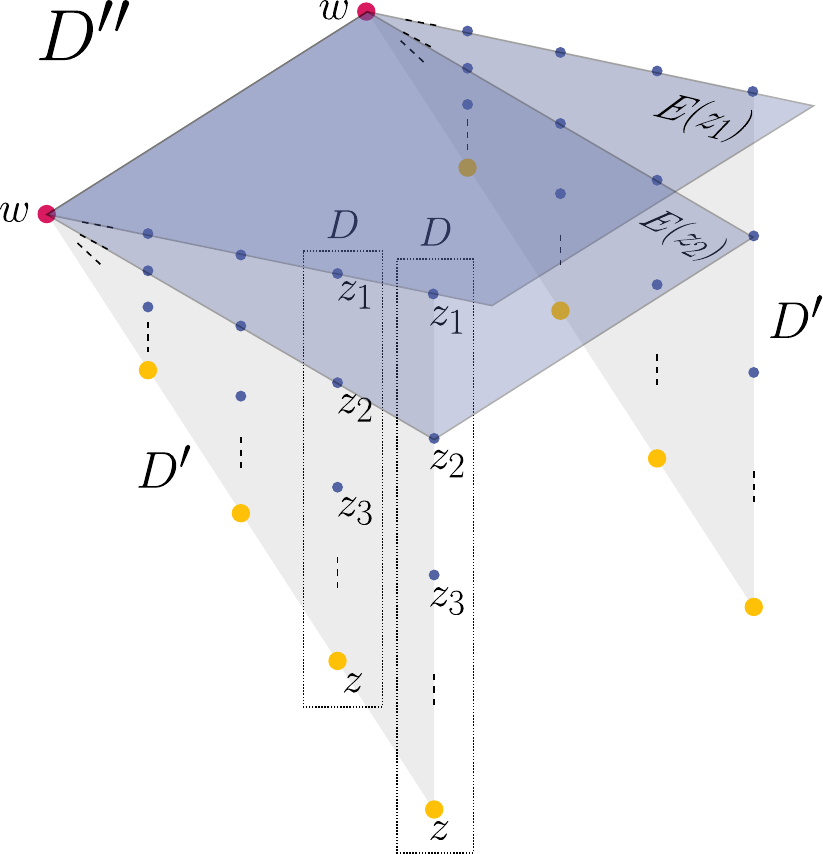}
    \caption{Another representation of $D''$ in \Cref{exa:nontame_locCB}. Each gray triangle represents a copy of $D'$, and each blue-shaded quadrilateral represents the equivalence class of $z_1$ and $z_2$, respectively.  By taking $X=\{w\}$ in one copy of $D'$ and $\{U_n\}_{n \ge 1}$ to be nested sets in the same copy of $D'$ shrinking to $X$, we show that $D''$ is of limit type.}
    \label{fig:nontame_locCB_2}
    
\end{figure}

\end{EXA}

Although the end space $D$ used in the examples above is countable, we could replace $D$ with an uncountable end space, provided it still has a unique maximal end and lacks any stable neighborhood. For such an end space, the same conclusions would hold, since the results of Jiménez-Rolland--Morales \cite{rolland2023large} apply and determine the (non-)CB-ness in an analogous way.

\subsection{Summary}
\label{subsec:examples}

\Cref{tab:guide} summarizes some important properties of the examples discussed in this paper.

\colorlet{clrY}{green!20}
\colorlet{clrN}{yellow!20}

\begin{table}[h!]
\renewcommand{\arraystretch}{1.2}
\makebox[\textwidth][c]{
\begin{tabularx}{1.023\textwidth}{>{\centering\arraybackslash}m{1.2cm}|m{4.2cm}| >{\centering\arraybackslash}m{1cm} | >{\centering\arraybackslash}m{1cm} | >{\centering\arraybackslash}m{1cm} | >{\centering\arraybackslash}m{1cm} | >{\centering\arraybackslash}m{1cm} }
 \Xhline{0.5ex}
 & {\bf Example} & {\bf CB} & {\bf loc. CB} & {\bf Tame} & {\bf Finite Rank} & {\bf Not Limit Type} \\ \hline
 \multirow{3}{*}{\centering\rotatebox{90}{\small{\textbf{CB-generated\,}}}} 
 & Loch Ness (\Cref{fig:lochandlad})        & Y\cellcolor{clrY} & Y\cellcolor{clrY} & Y\cellcolor{clrY} & Y\cellcolor{clrY} & Y\cellcolor{clrY} \\ \cline{2-7}
 &\Cref{exa:partialorder}, \Cref{exa:np_finite}, \Cref{exa:MinfPfin}, \Cref{exa:nmp_multgenshift}  & N\cellcolor{clrN} & Y\cellcolor{clrY} & Y\cellcolor{clrY} & Y\cellcolor{clrY} & Y\cellcolor{clrY} \\ \cline{2-7}
 & $D'$ (\Cref{exa:nontame_CB})                        & Y\cellcolor{clrY} & Y\cellcolor{clrY} & N\cellcolor{clrN} & Y\cellcolor{clrY} & Y\cellcolor{clrY} \\ \hline\hline
 \multirow{5}{*}{\centering\rotatebox{90}{\small{\textbf{not CB-gen'd}}}} 
 & Punctured Loch Ness (\Cref{fig:LNMonsterWithPunctures})        & N\cellcolor{clrN} & N\cellcolor{clrN} & Y\cellcolor{clrY} & Y\cellcolor{clrY} & Y\cellcolor{clrY} \\ \cline{2-7}
 & \Cref{exa:infgenshift}                    & N\cellcolor{clrN} & N\cellcolor{clrN} & Y\cellcolor{clrY} & N\cellcolor{clrN} & N\cellcolor{clrN} \\ \cline{2-7}
 & \Cref{exa:kanagawa}          & N\cellcolor{clrN} & Y\cellcolor{clrY} & Y\cellcolor{clrY} & N\cellcolor{clrN} & N\cellcolor{clrN} \\ \cline{2-7}
 & $D''$ (\Cref{exa:nontame_locCB})                    & N\cellcolor{clrN} & Y\cellcolor{clrY} & N\cellcolor{clrN} & N\cellcolor{clrN} & N\cellcolor{clrN} \\ \cline{2-7}
 & $D$ (\Cref{exa:nontame_basic})       & N\cellcolor{clrN} & N\cellcolor{clrN} & N\cellcolor{clrN} & Y\cellcolor{clrY} & Y\cellcolor{clrY} \\
 \Xhline{0.5ex}
\end{tabularx}
}
\caption{Summary of the features of the surfaces in all examples. The top portion of the table summarizes examples with CB-generated mapping class groups, while the bottom portion covers examples with non-CB-generated mapping class groups. The properties shown here are 1)CB mapping class group, 2)locally CB mapping class group, 3)tame end space, 4)finite rank mapping class group, and 5)the end space $E(S)$ is not of limit type. We marked Y(for Yes) if the surface has the corresponding property, and N(for No) otherwise.}
\label{tab:guide}
\end{table}
The examples listed in \Cref{tab:guide} are not exhaustive. In particular, the top portion of the table would fully capture all possible CB-generated examples if, in addition, we had an example of a surface $S$ satisfying the following properties: $\Map(S)$ is not CB but CB-generated(and hence locally CB); $S$ does \emph{not} have a tame end space; $\Map(S)$ has finite rank; $E(S)$ is not of limit type.  The authors of this article are unaware of such an example.  Based solely on the information and tools in \cite{mann2023large}, we cannol rule it out; however, from a discussion in Section 6.3, it follows:

\begin{PROP}[{cf. \cite[Section 6.3]{mann2023large}}]\label{thm:Bgenimpliesmaxendstab}
    Let $S$ be an infinite type surface such that $\Map(S)$ is CB-generated but not globally CB. Then every maximal end of $S$ is stable.
\end{PROP}
\begin{proof}
    By \Cref{prop:CBimplications}, $\Map(S)$ is locally CB. Then \Cref{thm:locCB} tells us that there exists a finite type subsurface $K\subset S$ which fits the criteria of the theorem. In particular, $K$ partitions $E(S\setminus \Int(K))=\left(\bigsqcup_{A\in\mathcal{A}}A\right)\sqcup\left(\bigsqcup_{P\in\mathcal{P}}P\right)$ so that for each $A\in\mathcal A$, $\mathcal M(A)\subset\mathcal M(E(S\setminus \Int(K)))$ and $\mathcal M(E(S\setminus \Int(K)))=\bigsqcup_{A\in\mathcal A}\mathcal M(A)$. However, $K$ can only contain finitely many isolated planar ends. 
    
    In particular, if $x\in E(K)$, then $x$ is an isolated planar end. Further, since $x\in \mathcal M(E(S))$, $x$ is the only isolated planar end in $E(S)$. Pick any neighborhood of $x$ which is homeomorphic to a disk with a puncture. This neighborhood is stable, so $x$ is a stable end.

    If $x\notin E(K)$, then $x\in \mathcal M(E(S\setminus \Int(K)))$, so there exists some $A\in \mathcal A$ such that $x\in \mathcal M(A)$. Let $N\subset S$ be the connected component of $S\setminus K$ such that $E(N)=A$. Then $N$ is a neighborhood of $x$. Let $V\subset N$ be a neighborhood of $x$. By Condition (2c) of \Cref{thm:locCB}, there exists $f_V\in\Homeo(S)$ such that $N\subset f_V(V)$, so $V$ contains a homeomorphic copy of $N$. Thus, $N$ is stable and $x$ is a stable end.
\end{proof}

\begin{COR}\label{cor:exampletolookfor}
Let $S$ be a surface whose end space is not tame and not of limit type. Assume $\Map(S)$ has finite rank, is not CB, but is CB-generated. Then $S$ has a non-stable immediate predecessor of a maximal end.
\end{COR}

The bottom half of \Cref{tab:guide} is also incomplete. 
In addition to the examples so far discussed, we added the following surface as an example in the table:

\begin{EXA}[Tame, locally CB, not CB-generated]\label{exa:kanagawa}

Let $S'$ be a surface with endspace $(E(S'), E_G(S')) \cong (\omega^\omega + 1, \emptyset)$. The surface $S'$ can be visualized as ``half'' of the surface in \Cref{fig:kanagawa} (see the caption for an explanation).  Because $E(S')$ is self-similar, $\Map(S')$ is globally CB by \Cref{thm:CB}.  By \Cref{thm:locCB}(2) taking connect-sums of self-similar surfaces results in a surface that is locally CB.  Indeed, for the surface $S = S' \# S'$ as illustrated in \Cref{fig:kanagawa}, $\Map(S)$ is locally CB.  To see this, let $K \subset S$ be an annulus in the center of \Cref{fig:kanagawa} between the two ends locally homeomorphic to $\omega + 1$.   The end spaces of the resulting components of $\Int(S \setminus K)$ are both self-similar.  There are no $P$ sets.  Additionally, the small zoom condition of \Cref{thm:locCB} is satisfied, so $\Map(S)$ is locally CB. 
\begin{figure}[ht!]
    \centering
    \scalebox{.94}{
    \begin{overpic}[width=\textwidth]{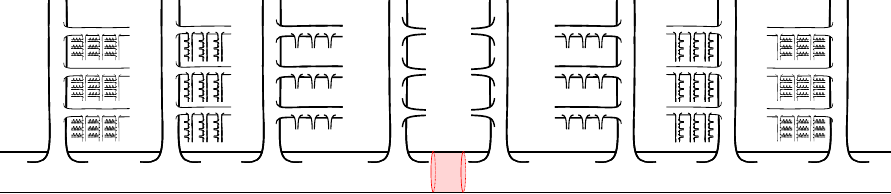}
        \put(-6,1.5){$e_1 \cdots$}
        \put(99,1.5){$\cdots e_2$}
        \put(48.9,1.5){\small\color{red}$K$}
    \end{overpic}
    }
    \caption{\Cref{exa:kanagawa}. This surface
    is tame and has locally CB but not CB-generated mapping class group. $S'$ is a single component of $S\setminus K$ with boundary capped off by a disk. }
    \label{fig:kanagawa}
\end{figure} 
Interestingly, the surface is also tame, but \emph{not} CB-generated.  The two ends $e_1$ and $e_2$ are maximal and stable since all sufficiently
  small neighborhoods of them are homeomorphic to one
another.  Furthermore, since $S$ has ends homeomorphic to $\omega^n + 1$ for all
positive $n$, neither $e_1$ nor $e_2$ has any immediate predecessors.  Therefore, $S$ is tame.  
$\Map(S)$ is not CB-generated since $S$ has infinite rank, in particular $E(z_n) = \omega^n + 1$ for all $n$ each with accumulating points $e_1$ and $e_2$. 

\end{EXA}

\vspace{.5cm}

Of the 16 possible combinations of Y/N in the bottom section of \Cref{tab:guide}, we have provided examples of only 5.  Some combinations are not possible.  For example, there is no (N,Y,Y,Y,Y) row; a surface which has tame end space, has end space not of limit type, has finite rank mapping class group, but whose mapping class group is locally CB but not CB-generated or globally CB by \Cref{thm:CBgen}.  Examples or non-existence of such examples in the remaining 10 cases are not known to the authors.

\bibliographystyle{plain}
\bibliography{bib}

\begin{thebibliography}{10}

\bibitem{aramayona2020big}
Javier Aramayona and Nicholas~G. Vlamis.
\newblock Big mapping class groups: an overview.
\newblock In {\em In the tradition of {T}hurston---geometry and topology},
  pages 459--496. Springer, Cham, [2020] \copyright 2020.

\bibitem{branman2024graphical}
Beth Branman, George Domat, Hannah Hoganson, and Robert~Alonzo Lyman.
\newblock Graphical models for topological groups: A case study on countable
  stone spaces.
\newblock {\em arXiv preprint arXiv:2411.15337}, 2024.

\bibitem{dehn1938gruppe}
Max Dehn.
\newblock Die gruppe der abbildungsklassen: Das arithmetische feld auf
  fl{\"a}chen.
\newblock {\em Acta mathematica}, 69(1):135--206, 1938.

\bibitem{Domat_2023}
George Domat, Hannah Hoganson, and Sanghoon Kwak.
\newblock Coarse geometry of pure mapping class groups of infinite graphs.
\newblock {\em Advances in Mathematics}, 413:108836, January 2023.

\bibitem{domat2023generatingsetsalgebraicproperties}
George Domat, Hannah Hoganson, and Sanghoon Kwak.
\newblock Generating sets and algebraic properties of pure mapping class groups
  of infinite graphs.
\newblock {\em arXiv preprint arXiv:2309.07885}, 2023.
\newblock To appear in Annales Henri Lebesgue.

\bibitem{farb2011primer}
Benson Farb and Dan Margalit.
\newblock {\em A primer on mapping class groups}, volume~49.
\newblock Princeton university press, 2011.

\bibitem{hill2023largescale}
Thomas Hill.
\newblock Large-scale geometry of pure mapping class groups of infinite-type
  surfaces.
\newblock {\em Proceedings of the American Mathematical Society},
  153(06):2667--2680, 2025.

\bibitem{kerekjarto1923vorlesungen}
B~v Ker{\'e}kj{\'a}rt{\'o}.
\newblock Vorlesungen {\"u}ber topologie: I, fl{\"a}chentopologie.
\newblock 1923.

\bibitem{lickorish1964finite}
William~BR Lickorish.
\newblock A finite set of generators for the homeotopy group of a 2-manifold.
\newblock In {\em Mathematical Proceedings of the Cambridge Philosophical
  Society}, volume~60, pages 769--778. Cambridge University Press, 1964.

\bibitem{mann2023large}
Kathryn Mann and Kasra Rafi.
\newblock Large-scale geometry of big mapping class groups.
\newblock {\em Geom. Topol.}, 27(6):2237--2296, 2023.

\bibitem{mann2022results}
Kathryn Mann and Kasra Rafi.
\newblock Two results on end spaces of infinite type surfaces.
\newblock {\em Michigan Mathematical Journal}, 74(5):1109--1116, 2024.

\bibitem{mazurkiewicz1920contribution}
Stefan Mazurkiewicz and Wac{\l}aw Sierpi{\'n}ski.
\newblock Contribution {\`a} la topologie des ensembles d{\'e}nombrables.
\newblock {\em Fundamenta Mathematicae}, 1(1):17--27, 1920.

\bibitem{patel2018algebraic}
Priyam Patel and Nicholas Vlamis.
\newblock Algebraic and topological properties of big mapping class groups.
\newblock {\em Algebraic \& Geometric Topology}, 18(7):4109--4142, 2018.

\bibitem{Richards1963OnTC}
Ian Richards.
\newblock On the classification of noncompact surfaces.
\newblock {\em Trans. Amer. Math. Soc.}, 106:259--269, 1963.

\bibitem{rolland2023large}
Rita~Jiménez Rolland and Israel Morales.
\newblock On the large scale geometry of big mapping class groups of surfaces
  with a unique maximal end.
\newblock {\em arXiv preprint arXiv:2309.05820}, 2023.

\bibitem{rosendal2013global}
Christian Rosendal.
\newblock Global and local boundedness of polish groups.
\newblock {\em Indiana University Mathematics Journal}, pages 1621--1678, 2013.

\bibitem{rosendal2014largescalegeometrymetrisable}
Christian Rosendal.
\newblock Large scale geometry of metrisable groups.
\newblock {\em arXiv preprint arXiv:1403.3106}, 2014.

\bibitem{rosendal2021coarse}
Christian Rosendal.
\newblock {\em Coarse geometry of topological groups}, volume 223 of {\em
  Cambridge Tracts in Mathematics}.
\newblock Cambridge University Press, Cambridge, 2022.

\bibitem{udall2024spherecomplexlocallyfinite}
Brian Udall.
\newblock The sphere complex of a locally finite graph.
\newblock {\em arXiv preprint arXiv:2407.07976}, 2024.

\bibitem{vlamis2019notes}
Nicholas~G Vlamis.
\newblock Notes on the topology of mapping class groups.
\newblock {\em preprint}, 2019.
\newblock \url{https://vlamis.nyc/assets/files/AIM_Notes.pdf}.

\end{thebibliography}

\end{document}